\documentclass[11pt]{article}

\usepackage{amsmath}
\usepackage{amsthm}
\usepackage{amsfonts}
\usepackage{bbm}

\usepackage{graphicx}

\usepackage{xcolor}

\newtheorem{thm}{Theorem}
\newtheorem{lem}{Lemma}
\newtheorem{cor}{Corollary}
\theoremstyle{definition}
\newtheorem{defn}{Definition}

\theoremstyle{remark}
\newtheorem{remk}{Remark}

\newcommand{\1}{1\!\!\!\!\!\;{\rm I}}
\newcommand{\mbR}{{\mathbb R}}

\newcommand{\tp}{\overset{P}{\to}}

\newcommand{\sign}{\mathop{\rm sgn}}

\newcommand{\ve}{\varepsilon}

\newcommand{\be}{\begin{equation}}
\newcommand{\ee}{\end{equation}}

\begin{document}
\title{On a Brownian motion with a hard membrane }

\author{Vidyadhar Mandrekar\footnote{Michigan State University, East Lansing, USA
\newline e-mail: mandrekar@stt.msu.edu} \ \ and \ \
Andrey Pilipenko\footnote{Institute of Mathematics, National
Academy of Sciences of Ukraine, Kyiv, Ukraine,
\newline e-mail: pilipenko.ay@yandex.ua} }

\maketitle
\begin{abstract}
Local perturbations of a Brownian motion are considered. As a limit we obtain a non-Markov process that behaves as a reflecting Brownian motion on the positive half line until its local time at zero reaches some exponential level, then changes a sign and behaves as a reflecting Brownian motion on the negative half line until some stopping time,  etc.



\end{abstract}

% \textbf{Keywords:} ???}

\section{Introduction}
Consider a sequence of  SDEs
\be\label{eq_main}
dX_\ve(t)=a_\ve(X_{\ve}(t))dt +dw(t),\ t\geq 0,\\
X_{\ve}(0)=x,
\ee
where $w$ is a Wiener process.

We assume that  $a_\ve$ is an integrable function; this ensures existence and  uniqueness of a strong solution to this SDE \cite{EngelbertSchmidt}, Theorem 4.53.

We also will suppose  that the support of $a_\ve$ is contained in $[-\ve, \ve]$; $X_\ve $ will be interpreted as a local perturbation of a Brownian motion.

Condition $\mbox{ supp} (a_\ve) \subset[-\ve,\ve] $ ensures the weak relative compactness in the space of continuous functions of  $\{X_{\ve_n}\}$ for any sequence
   $\ve_n\to 0$ as $n\to\infty.$ { Indeed, if $\mbox{ supp} (a_\ve) \subset[-\ve,\ve] $, then it can be seen that $\omega_{X_\ve}(\delta)\leq 2\omega_{w}(\delta)+2\ve$ for any $\ve>0, \delta>0$, where $\omega_{g}(\delta)=\sup_{s,t\in[0,T]; |s-t|\leq\delta}|g(t)-g(s)|$ is the modulus of continuity (see Lemma 3 below).} The  aim of this paper is to discuss  possible  limits of   $\{X_{\ve }\}$  as $\ve\to 0+$.

If $a_\ve(x)=\ve^{-1}a(\ve^{-1}x),$ then $a_\ve$ converges in the generalized sense to $\alpha \delta,$ where $\alpha=\int_\mbR a(x)dx$ and $\delta$ is
the Dirac delta function at $0.$ In this case \cite[Theorem 8]{Portenko}, \cite[Proposition 8]{Lejay} we have convergence in distribution in the space of continuous functions
$$
X_\ve\Rightarrow w_\gamma,
$$
where $\gamma=\tanh \alpha,$ $w_\gamma$ is a skew Brownian motion, i.e., a continuous Markov process with transition density
$$
p_t(x,y) = \varphi_t (x-y) + \gamma\sign(y)\varphi_t (|x|+|y|),\
x,y \in \mathbb{R},
$$
  $\varphi_t(x) = \frac{1}{\sqrt{2\pi t}} e^{-x^2/2t}$ is the density of the normal distribution  $N(0,t)$.

If $\gamma=1$ (or $\gamma=-1$), then $w_\gamma$ is a  Brownian motion with reflection into the positive (or negative) half-line.

Assume that $\sign( x) a_\ve(x)\geq 0.$ Then the drift term pushes up when $X_\ve$ is on the positive half line and pushes down when
$X_\ve$ is on the negative half line. If the limit of sequence $\1_{x\geq 0}a_\ve(x)$ as $\ve\to 0+$ is ``greater than'' delta function, then any limit of $X_\ve$ cannot cross through  zero and consequently the limit will be a reflecting Brownian motion. %, see for example results on penalization method \cite{Slominski, PilipenkoUMZh}.

Note that the skew Brownian  with $|\gamma|<1$ has both positive and negative excursions in any neighborhood of hitting 0 with probability 1.
Reflected Brownian motion ($ \gamma =1$) does not cross zero if it starts from $x\geq 0$; otherwise if $x<0$, then  it crosses 0 immediately after the hitting.
We find a situation when a limit of $\{X_\ve\}$  is an intermediate regime between a reflecting case and a skew Brownian motion.
The limit process will be a reflecting Brownian motion in some half line until its local time reaches an exponential random variable. Then it behaves as a Brownian motion with reflection into another half line until its local time reaches another independent exponential random variable, etc. We call such process a Brownian motion with a hard membrane.  
The corresponding definitions are given in \S\ref{Sect1}. We prove the general convergence result in \S  \ref{Sect2}.
As an example we discuss  in \S
\ref{Sect3} the case $a_\ve(x)=L_\ve \ve^{-1}a(\ve^{-1} x),\  \mbox{supp} (a) \subset [-1,1]$,where $L_\ve\to\infty$ as $\ve\to 0+$.

{  The Brownian motion with a hard membrane can be also obtained as a scaling limit of the Lorentz process in a strip with a reflecting wall at the origin that has small shrinking  holes \cite{NandoriSzasz}.}

\section{Definitions. Reflecting Brownian motion. Brownian motion with a hard membrane. }\label{Sect1}

Recall the definition and properties of the Skorokhod reflection problem, see for ex. \cite{Pilipenko_RSDE}.
\begin{defn}
\label{defn1}
Let  $f\in C([0, T]), f(0)\geq0$. A pair of continuous functions $g$ and $l$ is said to be a solution of the Skorokhod problem for $f$ if

{\bf S1.} $g(t)\geq0, \ t\in[0, T];$

{\bf S2.} $g(t)=f(t)+l(t), \ t\in[0, T];$

{\bf S3.} $l(0)=0, \ l$ is non-decreasing;

{\bf S4.} $\int\limits^T_0\1_{g(s)>0}dl(s)=0.$
\end{defn}
It is well known that there exists a unique solution to the Skorokhod problem and the solution is given by the formula
\begin{equation}
\label{eq2.1}
l(t)=-\min_{s\in[0, t]}(f(s)\wedge0)=\max_{s\in[0, t]}(-f(s)\vee0),
\end{equation}
\begin{equation}
\label{eq2.2}
g(t)=f(t)+l(t)=f(t)-\min_{s\in[0, t]}(f(s)\wedge0).
\end{equation}

We  will say that $g$ is a process  reflecting into a positive half line and denote it by $f^{refl, +}.$
\begin{remk}\label{remSkor}
If $f(0)<0,$ then
we set by definition  $f^{refl, +}(t):=f(t)$ until the instant $\zeta_0$ of hitting
 $0,$ and $f^{refl, +}(t):=f(t)-\min_{s\in[\zeta_0, t]} f(s)$ for $t\geq \zeta_0.$
 \end{remk}
 The reflection problem with reflection into the negative half line is constructed similarly, $g(t):=f(t)-l(t),$ where $l$ is also non-decreasing.
In this case denote $g$ by $f^{refl, -}.$

Let $w(t), t\geq 0$ be a Brownian motion started at $x$, $w^{refl,+}$ and $w^{refl, -}$ be reflecting Brownian motions with reflection at $0$ into the positive and negative half lines, correspondingly. It is known that the process $l(t)$ is the two-sided local time  of $w^{refl, \pm}$ at   $0$ defined by\\
$\lim_{\ve\to 0+}(2\ve)^{-1}\int_0^t\1_{[-\ve,\ve]}(w^{refl, \pm}(s))ds.$

Consider two sequences of exponential random variables $\{\xi_k^+\}$ and  $\{\xi_k^-\}$ with parameters $\alpha^\pm$, respectively. Suppose that all random variables  $\{\xi_k^\pm\}$ and the Brownian motion $w(t), t\geq 0$ are mutually independent.

Assume that $w(0)=x>0.$ The Brownian motion with a hard membrane and parameters of permeability $\alpha^\pm$ is constructed in the following way.
$$
w^{hard}(t):=w^{refl, +}(t)\ \mbox{if } l(t)\leq {\xi}^+_1.
$$
At the moment $\zeta^+_1:=\inf\{t\geq 0\ | \ l(t){\geq}  {\xi}^+_1\}$   the process $w^{hard}$ changes orientation. It  reflects into negative half line until
the moment when its local time after $\zeta_1^+$ reaches the level
$ {\xi}_1^-$:
$$
w^{hard}(t):= {w(t)-w(\zeta_1^+)-(l(t)-l(\zeta_1^+))}=w(t)-w(\zeta_1^+)-\max_{s\in[\zeta_1^+, t]}( w(s) -w(\zeta_1^+))
$$
 for $t\leq \inf\{s\geq \zeta_1^+\ : \ (l(s)-l(\zeta_1^+))\geq \xi_1^-\}= \inf\{t\geq \zeta_1^+\ : \ \max_{s\in[\zeta_1^+, t]}( w(s) -w(\zeta_1^+))\geq \xi_1^-\}.$

Denote $\zeta^-_1:=\inf\{t\geq \zeta_1^+\ : \max_{s\in[\zeta_1^+, t]}( w(s) -w(\zeta_1^+))\geq \xi_1^-\}$. At the moment $\zeta^-_1$ changes its orientation again and reflects into the positive half line, and so on.
The formal equation for $w^{hard}$ is the following
$$
w^{hard}(t)= w(t)+\int_0^t \left(\1_{l(s)\in \cup [\xi_1^++\xi_1^-+...+\xi_k^++\xi_k^-, \xi_1^++\xi_1^-+...+\xi_k^++\xi_k^-  +\xi_{k+1}^+)}- \right.
$$
$$
\left.- \1_{l(s)\in \cup [\xi_1^++\xi_1^-+...+\xi_k^+ , \xi_1^++\xi_1^-+...+\xi_k^++\xi_k^- )}\right)dl( {s}).
$$
If $x=w(0)<0,$ then  process $w^{hard}(t), t\geq 0$ is constructed similarly.

 In case $w(0)=0,$   we have to ``attach'' the initial direction of reflection at zero.
Denote the direction of reflection of $w^{hard}$ at time $t$ by $\sign w^{hard}(t)\in\{-1, 1\}.$ Note that
$\sign w^{hard}(t)=1$ if $ w^{hard}(t)>0$ and $\sign w^{hard}(t)=-1$ if $ w^{hard}(t)<0$; $\sign w^{hard}(t)$ may change a sign only at instants
$\xi_1^++\xi_1^-+...+\xi_k^++\xi_k^-$ or $\xi_1^++\xi_1^-+...+\xi_k^+.$ We will always select a cadlag
modification for $\sign w^{hard}(t), t\geq 0  $. It can be seen that the pair $(w^{hard}(t),\sign w^{hard}(t))$ is a strong Markov process on $\mbR\times \{-1, 1\}.$ We can informally consider $w^{hard}(t), \ t\geq 0$ as the strong Markov process on the set $(-\infty, 0-]\cup [0+,\infty),$ where $0\pm$ means direction of the reflection at when $w(0)=0$.

\section{General conditions of convergence}\label{Sect2}
We assume that $\mbox{supp} a_\ve\subset  {[-\ve,\ve]}, \ \sign(x) a_\ve(x)\geq 0,\  {a_\ve}\in L_1(\mbR).$

Let us suppose for simplicity that $X_\ve(0)=x>0.$
 Introduce a sequence of stopping times
 $$
 \sigma^{(\ve)}_0:=0;
 $$
$$
 \tau^{(\ve)}_{n+1}:=\inf\{t\geq \sigma^{(\ve)}_{n}\ : \ X_{\ve}(t)=\ve\},
\ n\geq 0;
 $$
$$
 \sigma^{(\ve)}_{n }:=\inf\{t\geq \tau^{(\ve)}_{n}\ : \ X_{\ve}(t)=2\ve\},
\ n\geq 1.
 $$
 \begin{remk}
 For any $n\geq 1$ moments $\sigma^{(\ve)}_{n }, \tau^{(\ve)}_{n }$ are finite a.s., and
 $$
 \lim_{n\to\infty}\sigma^{(\ve)}_{n }= \lim_{n\to\infty} \tau^{(\ve)}_{n }=\infty \ \ \mbox{a.s.}
 $$
 \end{remk}
 Set
 $$
A_{\ve} (t):=\int_0^t\1_{ {s}\in \cup_k[ \sigma^{(\ve)}_{k }, \tau^{(\ve)}_{k +1}]}ds,
 $$
 $$
 A_\ve^{(-1)}(t): =\inf\{ s\geq 0\ :\ A_{\ve}(s)>t\},
 $$
 $$
 \bar X_\ve(t):=X_\ve(  A_\ve^{(-1)}(t)),\ \  {\bar w_\ve(t):= w(0)+ \int_0^{A_\ve^{(-1)}(t)} \1_{s\in \cup_k[ \sigma^{(\ve)}_{k }, \tau^{(\ve)}_{k +1}]}dw(s)}.
 $$
 {Observe that $\bar w_\ve(t), t\geq 0$ is a Wiener process.
}
 {
\begin{remk}
We will always assume that if $X(0)=x,$ then  $\{w(t)\}$ is a Wiener process also started from   $x.$
\end{remk}
}

{Let $\bar n_\ve(t)$ be the number of times when the process $\bar X_\ve(s), s\in [0,t]$ hits $\ve$,
$$
\bar n_\ve(t)=|\{s\in [0,t]\ : \ \bar X_\ve(s-)=\ve\}|.
$$
} 
  Informally, the behavior of $\bar X_\ve$ is the following. It moves as a Brownian motion until it hits $\ve.$ Next $\bar X_\ve$ immediately jumps to $2\ve$
  and moves as a Brownian motion (a  {vertical shift of}  $\bar w_\ve$) again until the second hitting $\ve, $ then jumps to $2\ve,$ etc.

\begin{figure}[h]
  % Requires \usepackage{graphicx}
    \caption{}
          \includegraphics[width=15 cm]{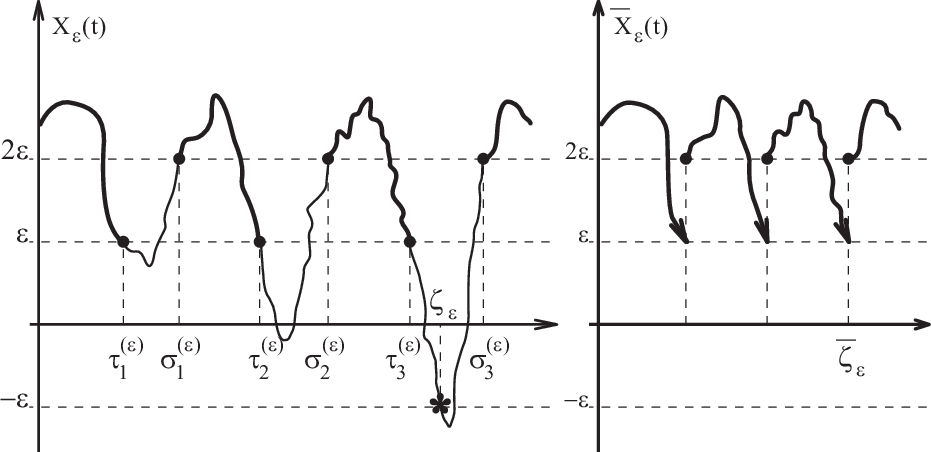}
    %\label{}
\end{figure}

%\vskip 10cm

 {
 \begin{lem}
 Assume that $X_\ve(0)=x,$  where $x>\ve$.
Then
\be\label{eq_X_ve}
\bar X_\ve(t)= {\bar w_\ve(t)}+\ve \bar n_\ve(t),
\ee
and
\be\label{eq_n_ve}
\bar n_\ve(t):=\left[- {\ve^{-1}}\min_{s\in[0,t]}\left( (\bar w_\ve(s)-2\ve)\wedge 0\right)\right]
\ee
where $[\, \cdot \;]$ is the integer part of a number.
\end{lem}
\begin{proof}
Formula \eqref{eq_X_ve} follows from the definition of $\bar X_\ve(t),$ $ \bar w_\ve(t),$ and $\bar n_\ve(t).$
Let $\bar \tau^{(\ve)}_k$ be the instant of  $k$-th jump of $\bar X_\ve.$  Observe that
\be\label{eq_tauk}
\bar \tau^{(\ve)}_k=\inf\{t\geq 0\ |\ \bar n_\ve(t)=k\}= \inf\{t\geq 0\ | \ \bar w_\ve(t) = (2-k)\ve\}, \ k\geq 1.
\ee
and
\be\label{eq_tauk1}
\bar n_\ve(t)=k, t\in [\bar \tau^{(\ve)}_k,\bar \tau^{(\ve)}_{k+1}).
\ee
Formulas \eqref{eq_tauk}, \eqref{eq_tauk1} imply \eqref{eq_n_ve}.
\end{proof}
}

Set
$$
\zeta_\ve:=\inf\{ t\geq 0:\ \ X_{\ve}(t)=-\ve\},
$$
 $$
 \bar G_\ve=\inf\left\{ k\geq 1\ :\ \sigma^{(\ve)}_k> \zeta_\ve\right\}.
 $$

Define

 $$\bar \zeta_\ve=A_{\ve} (\zeta_\ve) =\int_0^{\zeta_\ve} \1_{s\in \cup_k[ \sigma^{(\ve)}_{k }, \tau^{(\ve)}_{k +1}]}ds=\sum_{k=0}^{\bar G_\ve-1}(\tau^{(\ve)}_{k+1}-\sigma^{(\ve)}_k).$$

  Set
\be\label{eq_p+}
  p_\ve^+= P\Big( \inf\{ t\geq 0: \ X_\ve(t)=-\ve\}<\inf\{t\geq 0:\ X_\ve(t)=2\ve\} \  | \ X_\ve(0)=\ve \Big);
\ee
  $p_\ve^+$
is the probability for the process $X_\ve $  to reach $-\ve$ before $2\ve$ given $ X_\ve(0)=\ve.$

 \begin{lem}\label{lem_equal}
 Assume that $X_\ve(0)=x,$  where $x>\ve$.

Set
$$
\tilde X_\ve(t)=  {w(t)}+\ve \tilde n_\ve(t),
$$
where
$$
\tilde n_\ve(t):=\left[- {\ve^{-1}}\min_{s\in[0,t]}\left( (w(s)-2\ve)\wedge 0\right)\right].
$$

Let $\tilde G_\ve^+$ be a geometrical random variable with parameter $p_\ve^+,$
 $$
 P(\tilde G^+_\ve=n)=(1-p_\ve^+)^{n-1}p_\ve^+, \ n\geq 1;
 $$
 and $\tilde G^+_\ve$  is independent of a Wiener process
 $w(t), t\geq 0.$

 Denote $\tilde \zeta_\ve:=\inf\{t\geq 0:\ \tilde n_\ve(t)\geq \tilde G_\ve^+\}.$

 Then the distributions   in $D([0,\infty))^2\times C([0,\infty))\times \mbR$ of quadruples
 \newline
 $
(\bar X_\ve(\cdot), \bar n_\ve(\cdot),  {\bar w_\ve(\cdot)}, \bar \zeta_\ve)
$
 and
 $(\tilde X_\ve(\cdot), \tilde n_\ve(\cdot), w(\cdot), \tilde \zeta_\ve)$
   are equal.
 \end{lem}
 
\begin{proof}
Observe that the following $\sigma$-algebras are equal
$$
\sigma(\bar X_\ve(t), \bar n_\ve(t), \bar w_\ve(t), t\geq 0)= \sigma(  {\bar w_\ve(t)}, t\geq 0)= 
$$
$$
\sigma \left( \int_0^{A_\ve^{(-1)}(t)} \1_{s\in \cup_k[ \sigma^{(\ve)}_{k }, \tau^{(\ve)}_{k+1 }]}dw(s), t\geq 0\right) 
=\sigma \left( \int_0^{t} \1_{s\in \cup_k[ \sigma^{(\ve)}_{k }, \tau^{(\ve)}_{k+1 }]}dw(s), t\geq 0\right)=
$$
$$
=\sigma \left( w(t)-w(\sigma^{(\ve)}_{k }), t\in  [ \sigma^{(\ve)}_{k }, \tau^{(\ve)}_{k+1}], k\geq 0\right).
$$
 Random variable $\bar G^+_\ve$ is measurable with respect to 
$$
\sigma\left(X_\ve(t), t\in \cup_k[ \tau^{(\ve)}_{k }, \sigma^{(\ve)}_{k }]\right)=\sigma\left(X_\ve(t)-\ve, t\in \cup_k[ \tau^{(\ve)}_{k }, \sigma^{(\ve)}_{k }]\right)=
$$
$$
\sigma\left(X_\ve(t)-X_\ve(\tau^{(\ve)}_{k }), t\in [ \tau^{(\ve)}_{k }, \sigma^{(\ve)}_{k }], k\geq 1\right).
$$
We have equality of  $\sigma$-algebras $\sigma\left(X_\ve(t)-X_\ve(\tau^{(\ve)}_{k }), t\in [ \tau^{(\ve)}_{k }, \sigma^{(\ve)}_{k }], k\geq 1\right)
$ and 
\newline
$\sigma \left( w(t)-w( \tau^{(\ve)}_{k }),    t\in  [ \tau^{(\ve)}_{k }, \sigma^{(\ve)}_{k }], k\geq 1\right)$
%=\sigma \left( \int_0^{t} \1_{s\in \cup_k[ \sigma^{(\ve)}_{k }, \tau^{(\ve)}_{k+1 }]}dw(s), t\geq 0\right)$
 because  $X_\ve$ is a unique strong solution to \eqref{eq_main} (see Theorem 4.53 
\cite{EngelbertSchmidt}) and 
$X_\ve(\tau^{(\ve)}_{k })=\ve$.

Since
$\sigma \left( w(t)-w(\sigma^{(\ve)}_{k }), t\in  [ \sigma^{(\ve)}_{k }, \tau^{(\ve)}_{k+1}], k\geq 1\right)$
%$\sigma \left( \int_0^{t} \1_{s\in \cup_k[ \sigma^{(\ve)}_{k }, \tau^{(\ve)}_{k+1 }]}dw(s), t\geq 0\right)$
 is independent of 
 \newline
 $\sigma\left(X_\ve(t)-X_\ve(\tau^{(\ve)}_{k }), t\in [ \tau^{(\ve)}_{k }, \sigma^{(\ve)}_{k }], k\geq 1\right),$
%$\sigma \left( \int_0^{t} \1_{s\in \cup_k[\tau^{(\ve)}_{k }, \sigma^{(\ve)}_{k }]}dw(s), t\geq 0\right)$,
 random variable $\bar G^+_\ve$ is independent of 
\newline
$\sigma(\bar X_\ve(t), \bar n_\ve(t), \bar w_\ve(t), t\geq 0)$.

By the strong Markov property of $X_\ve$ we have
$$
P(\bar G^+_\ve>n+1\  | \ \bar G^+_\ve>n)= P(X_\ve(z)\neq -\ve, z\in [\tau^{(\ve)}_{n+1}, \sigma^{(\ve)}_{n+1}] \ | \ X_\ve(z)\neq -\ve, z\leq  \tau^{(\ve)}_{n+1} )=
$$
$$
= P\Big( \inf\{ t\geq 0: \ X_\ve(t)=-\ve\}\geq\inf\{t\geq 0:\ X_\ve(t)=2\ve\} \  | \ X_\ve(0)=\ve \Big)=1-p_\ve^+.
$$
Lemma \ref{lem_equal} is proved.
 \end{proof}
 
 \begin{thm}\label{thm1.1}  {Let $X_\ve(0)=x>0.$}  Assume that $\ve^{-1} p_\ve^+\to \alpha^+$ as $\ve\to 0+.$
 Let $Exp(\alpha^+)$  be exponential random variable that is independent of $w.$
Then $(\tilde X_\ve(\cdot), \ve\tilde n_\ve(\cdot), w(\cdot),  \tilde \zeta_\ve )$ converges weakly  to
$(w^{refl, +}(\cdot), l^{+}(\cdot), w(\cdot), \zeta )$ in $D([0,\infty))^2\times C([0,\infty))\times \mbR$,
where $(w^{refl,+},  l^+)$ is a  solution of the Skorokhod problem with reflection into the positive half line (see \S \ref{Sect1}), $\zeta=\inf\{t\geq 0: \ l^+(t)\geq Exp(\alpha^+)\}$.

 {Moreover
\be\label{eqStop}
(\tilde X_\ve(\cdot\wedge \tilde \zeta_\ve), \ve\tilde n_\ve(\cdot\wedge \tilde \zeta_\ve),  {\tilde w_\ve}(\cdot\wedge \tilde \zeta_\ve), \tilde \zeta_\ve )\Rightarrow
(w^{refl, +}(\cdot\wedge \zeta), l^{+}(\cdot\wedge \zeta), w(\cdot\wedge \zeta), \zeta)
\ee
as $\ve\to 0+$  in $D([0,\infty))^2\times C([0,\infty))\times \mbR$.
}
 \end{thm}
 \begin{proof}
It follows from Lemma \ref{lem_equal} and \eqref{eq2.1}, \eqref{eq2.2} that 
\be\label{eqUCF}
(\tilde X_\ve(\cdot), \ve\tilde n_\ve(\cdot), w(\cdot))\to (w^{refl, +}(\cdot), l^{+}(\cdot), w(\cdot)), \ve\to0+
\ee
uniformly on compact sets for almost all $\omega$.

 {  Since  $\ve^{-1}p_\ve\to \alpha, \ve\to0+,$ we have the weak convergence $\ve Geom(p_\ve^+)\Rightarrow Exp(\alpha^+)$ as $\ve\to 0+.$ By Skorokhod's representation theorem \cite[Theorem 6.7]{Billingsley}, there are copies of all processes given on the joint probability space (we will denote them by the same symbols) such that
\be\label{eqSkRep}
%$$
\ve\tilde G_\ve^+\to Exp(\alpha^+), \ve\to 0+ \ \mbox{almost surely},
%$$
\ee
 and $  \tilde G_\ve^+, Exp(\alpha^+)$
are independent of $w$.}

 {
It can be easily seen from the definition of $l^+$ that for any $c>0$ the instant $\zeta_c=\inf\{t\geq 0: \ l^+(t)\geq  {c}\}$
 is a point of increase of $l^+$ with probability 1, that is,
$$
P(\forall \delta>0: l^+(\zeta_c+\delta)>l^+(\zeta_c))=1.
$$
The independence of $Exp(\alpha{{^+}}) $
and $w$ yields that  $\zeta:=\inf\{t\geq 0: \ l^+(t)\geq Exp(\alpha^+)\}$ is also a point of increase of $l^+$ with probability 1. So, \eqref{eqUCF} and \eqref{eqSkRep} imply almost sure convergences
$$
\tilde \zeta_\ve=\inf\{t\geq 0:\ \ve\tilde n_\ve(t)\geq \ve\tilde G_\ve^+\}\to \inf\{t\geq 0: \ l^+(t)\geq Exp(\alpha^+)\}=\zeta,\ \ve\to0+,
$$  
and hence \eqref{eqStop}.
}
\end{proof}
\begin{cor}\label{cor1}
 Assume that $\ve^{-1} p_\ve^+\to \alpha^+$ as $\ve\to 0+.$
Then we have the weak convergence
$$
(\bar X_\ve(\cdot\wedge \bar \zeta_\ve), \ve\bar n_\ve(\cdot\wedge \bar \zeta_\ve),  {\bar w_\ve}(\cdot\wedge \bar \zeta_\ve), \bar \zeta_\ve )\Rightarrow
(w^{refl, +}(\cdot\wedge \zeta), l^{+}(\cdot\wedge \zeta), w(\cdot\wedge \zeta), \zeta)
$$
as $\ve\to 0+$  in $D([0,\infty))^2\times C([0,\infty))\times \mbR$,
where
$\zeta = \inf\{t\geq 0: \ l^+(t)\geq Exp(\alpha^+)\}$.
\end{cor}

\begin{thm}\label{thm_main1}
 Assume that $\ve^{-1} p_\ve^+\to \alpha^+$ as $\ve\to 0+.$
Then we have the weak convergence
$$
(X_\ve(\cdot\wedge \zeta_\ve),  w(\cdot\wedge \zeta_\ve), \zeta_\ve )\Rightarrow
(w^{refl, +}(\cdot\wedge \zeta), w(\cdot\wedge \zeta), \zeta),
$$
where
$\zeta_\ve =\inf\{ t\geq 0:\ \ X_{\ve}(t)=-\ve\},$ $\zeta=\inf\{t\geq 0: \ l^+(t)\geq Exp(\alpha^+)\}$.
\end{thm}
\begin{proof} {
Assume that $\zeta_\ve\in[0,T].$ Then by construction of $\bar \zeta_\ve, \bar X_\ve, \bar w_\ve$ we have
$$
|\bar \zeta_\ve-\zeta_\ve|\leq\int_0^{T}\1_{|X_\ve(s)|<2\ve}ds;
$$
\be\label{eq461}
\sup_{t\in[0,T]}|\bar X_\ve(t\wedge \bar \zeta_\ve)-X_\ve(t\wedge \zeta_\ve)|\leq 3\ve +\omega_{X_\ve,[0,T]}\left(\int_0^{\zeta_\ve}\1_{|X_\ve(s)|<2\ve}ds\right).
\ee
$$
\sup_{t\in[0,T]}|\bar w_\ve(t\wedge \bar \zeta_\ve)-w(t\wedge \zeta_\ve)|\leq 
\omega_{w,[0,T]}\left(\int_0^{T}\1_{|X_\ve(s)|<2\ve}ds\right)+\sup_{t\in[0,T]}|\int_0^{t\wedge \zeta_\ve} \1_{s\in \cup_k[ \tau^{(\ve)}_{k }, \sigma^{(\ve)}_{k }]}dw(s)|;
$$
Note that
$$
E \sup_{t\in[0,T]}\left(\int_0^{t\wedge \zeta_\ve}\1_{s\in \cup_k[ \tau^{(\ve)}_{k }, \sigma^{(\ve)}_{k }]}dw(s)\right)^2
 \leq4 E \int_0^{T\wedge \zeta_\ve}\1_{s\in \cup_k[ \tau^{(\ve)}_{k }, \sigma^{(\ve)}_{k }]}ds\leq 
4E\int_0^{T}\1_{|X_\ve(s) |<2\ve}ds.
$$
\begin{lem}\label{lemModCont}
For any $\ve>0, T>0$ we have
\be\label{eq444}
\omega_{X_\ve,[0,T]}\left(\delta \right)\leq 2\ve+2\omega_{w,[0,T]}\left(\delta \right),
\ee
where $\omega_{f,[0,T]}\left(\delta\right):=\sup_{s,t\in [0,T],\ |s-t|<\delta}|f(s)-f(t)|$ is the modulus of continuity of $f.$
\end{lem}
 Let $s\leq t, s,t\in[0,T].$ If $X_\ve(z)\notin [-\ve,\ve], z\in[s,t]$, then 
\be\label{eq450}
|X_\ve(t)-X_\ve(s)|=|w(t)-w(s)|\leq \omega_{w, [0,T]}(t-s). \
\ee
 Suppose that $X_\ve(z)\in [-\ve,\ve]$ for some $ z\in[s,t]$. Assume, for example, that $X_\ve(s)>\ve, X_\ve(t)<-\ve$. 
 All other cases are considered similarly. Denote $\tau:=\inf\{z\geq s | X_\ve(z)=\ve\},$
 $\sigma:=\sup\{z\leq t | X_{\ve}(z)=-\ve\}.$ Then 
\be\label{eq456}
 |X_\ve(s)-X_\ve(t)|=|X_\ve(s)-X_\ve(\tau)+ X_\ve(\tau)-X_\ve(\sigma)+X_\ve(\sigma)-X_\ve(t)|=
\ee
$$
 =|X_\ve(s)-X_\ve(\tau)|+ |X_\ve(\tau)-X_\ve(\sigma)|+|X_\ve(\sigma)-X_\ve(t)|=
 $$
$$
 =|X_\ve(s)-X_\ve(\tau)|+2\ve+|X_\ve(\sigma)-X_\ve(s)|=|w(s)-w(\tau)|+2\ve+|w(\sigma)-w(t)|\leq
 2\omega_{w, [0,T]}(t-s)+2\ve.
$$
Formulas \eqref{eq450}, \eqref{eq456} imply \eqref{eq444}. Lemma \ref{lemModCont} is proved.
}

It follows from \eqref{eq461} and \eqref{eq444} that
$$
\sup_{t\in[0,T]}|\bar X_\ve(t\wedge \bar \zeta_\ve)-X_\ve(t\wedge \zeta_\ve)| \leq 5\ve + 2\omega_{w,[0,T]}\left(\int_0^{T}\1_{|X_\ve(s)|<2\ve}ds\right).
$$

% If $\int_0^{T+1}\1_{|X(s)|<2\ve}ds<1,$ then by the construction of the process $X_\ve$,  for any $T>0$ we have the following inequalities
%$$
%|\bar \zeta_\ve-\zeta_\ve|< \omega_{X,[0,T]}\left(\int_0^{T+1}\1_{|X_\ve(s)|<2\ve}ds\right);
%$$
%$$
%\sup_{t\in[0,T]}|X_\ve(t\wedge \bar \zeta_\ve)-X_\ve(t\wedge \zeta_\ve)|\leq 3\ve +\omega_{X_\ve,[0,T]}\left(\int_0^{T+1}\1_{|X_\ve(s)|<2\ve}ds\right)\leq
%$$
%$$
%5\ve + 2\omega_{w,[0,T]}\left(\int_0^{T+1}\1_{|X_\ve(s)|<2\ve}ds\right).
%$$

The proof of the Theorem will follow from Corollary \ref{cor1} if we show that for any $T>0$
\be\label{eq_tanaka}
\int_0^{T}\1_{|X_\ve(s)|<2\ve}ds \tp 0, \ {\ve\to 0+}.
\ee
By Ito-Tanaka's formula we have
$$
|X_\ve(t)|=|x|+\int_0^t \sign (X_\ve(s)) a_\ve(X_\ve(s)) ds + B(t)+ l_\ve(t),
$$
where $B(t)=\int_0^t \sign (X_\ve(s)) d w(s)$ is a new Brownian motion, and $l_\ve$ is the local time of $X_\ve$ at $0.$ Notice that
$(|X_\ve(t)|, {l}_\ve(t)) $ is a solution of the Skorokhod reflection problem for the process $|x|+\int_0^t \sign (X_\ve(s)) a_\ve(X_\ve(s)) ds + B(t), \ t\geq 0.$
Since $\sign (x)a_\ve(x)\geq 0$, formulas \eqref{eq2.1}, \eqref{eq2.2} yield that $|X_\ve(t)|\geq B^{refl}(t):=|x|+B(t)-\min_{s\in [0,t]}((|x|+B(t))\wedge 0).$
Reflecting Brownian motion spends zero time at $0.$ This yields \eqref{eq_tanaka} and hence  completes the proof of the Theorem.
\end{proof}
Denote
$$
p_\ve^-:= P\Big( \inf\{ t\geq 0: \ X_\ve(t)=\ve\}<\inf\{t\geq 0:\ X_\ve(t)=-2\ve\} \  | \ X_\ve(0)=-\ve \Big).
$$
\begin{cor}
Assume that
\be \label{eq_4_1}
\lim_{\ve\to 0+}\ve^{-1} p_\ve^\pm=\alpha^\pm\in(0,\infty).
\ee
Then  $X_\ve$ converges in distribution to a Brownian motion with a hard membrane and parameters $\alpha^\pm.$
\end{cor}
\begin{remk}
If $\{\alpha^+=0 \ and \ \alpha^->0\}$ or $\{\alpha^+<\infty \ and \ \alpha^-=\infty\},$ then $X_\ve\Rightarrow w^{refl,+}$ as $\ve\to 0+.$
\end{remk}
\section{Example}\label{Sect3}
In this Section we give sufficient conditions that ensure  convergence of a sequence $\{X_\ve\}$ to a Brownian motion with a hard membrane.
Here we assume  that $a_\ve(x)=L_\ve\ve^{-1}a(\ve^{-1}x),$   $\mathrm{supp} a\subset [-1,1],\ \sign(x) a(x)\geq 0. $

We need to verify condition \eqref{eq_4_1}.

It is well known \cite{GS, Knight, ItoMcKean} that probabilities in \eqref{eq_4_1} are of the form
\be\label{eq_p}
p_\ve^+=\frac{s(2\ve)-s(\ve)}{s(2\ve)-s(-\ve)},
\ \
p_\ve^-=\frac{s(-2\ve)-s(-\ve)}{s(-2\ve)-s(\ve)},
\ee
where $s(x)=s_\ve(x)=\int_0^x\exp(-2\int_0^y a_\ve(z)dz)dy$ is the scale function of the diffusion $X_\ve.$

Denote $A(x):=\int_0^xa(y)dy.$ Then
$$
s_\ve(x)=\int_0^x\exp(-2\int_0^y a_\ve(z)dz)dy= \int_0^x\exp(-2\int_0^y L_\ve\ve^{-1}a(\ve^{-1}z)dz)dy=
$$
$$
=\int_0^x\exp(-2\int_0^{\ve^{-1}y} L_\ve a(u)du)dy=\int_0^x\exp(-2L_\ve A(\ve^{-1}y))dy=
$$
$$
=\ve\int_0^{\ve^{-1}x}\exp(-2L_\ve A(y))dy.
$$
So
$$
p_\ve^+=\frac{\int_1^{2}\exp(-2L_\ve A(y))dy}{\int_{-1}^{2}\exp(-2L_\ve A(y))dy}.
$$
Since $a(x)=0, x>1,$ we have for $x>0:$
$$
A(x)=A(1)=\int_0^1a(y)dy=\int_0^\infty a(y)dy=:A_+.
$$

 So
\be\label{eq_265}
\ve^{-1} p_\ve^+=
 \frac{ \ve^{{-1}} \exp(-2L_\ve A_+) }{ \int_{-1}^{2}\exp(-2L_\ve A(y))dy}.
\ee
\begin{thm}\label{thm2.1}
Assume that

\be\label{eq_as_A}
2A(x)\sim c_{\pm}|x|^\lambda, x\to 0 \pm,
\ee
 where $c_\pm>0,$ $\lambda>0.$

Then
\be\label{eq_lim_pe}
\lim_{\ve\to 0+}\ve^{-1}p^+_\ve=\alpha_+>0
\ee
 is equivalent to
the following condition
\be\label{eq_270}
L_\ve=(2A_+)^{-1}\Big(
\ln(\ve^{-1})+\lambda^{-1}\ln\ln(\ve^{-1})-
\ee
$$
-  (\ln(\alpha_+) + \ln(\Gamma(1+\lambda^{-1}))+  \lambda^{-1}\ln(2A_+)
+\ln(c_+^{-1/\lambda}+c_-^{-1/\lambda}) )
\Big) +o(1), \ \ve\to 0+.
$$
\end{thm}
\begin{proof}
 It can be seen that $p_\ve^+\to 0$ only if $L_\ve\to +\infty.$  The function $A$ attains a minimum at $0.$ So, for any $\delta>0$ formula \eqref{eq_as_A} yields (\cite{Pereval}, Ch.2, Lemma 1.3)
\be\label{eq_sim1}
\int_{-1}^{2}\exp(-2L_\ve A(y))dy\sim \int_{-\delta}^{\delta}\exp(-2L_\ve A(y))dy\sim
\ee
$$
\sim \int^{\delta}_0\exp(-L_\ve c_- y^\lambda) dy+\int^{\delta}_0\exp(-L_\ve c_+ y^\lambda) dy \sim \Gamma(1+\lambda^{-1})L_\ve^{-1/\lambda}(c_-^{-1/\lambda}+c_+^{-1/\lambda})
$$
as $\ve\to 0+.$

 Assume that \eqref{eq_lim_pe} is true. Then it follows from \eqref{eq_265} and \eqref{eq_sim1} that
\be\label{eq_298}
-\ln\alpha_+ +o(1)=-\ln\ve -2A_+L_\ve +  \lambda^{-1}\ln L_\ve -\ln\left( \Gamma(1+\lambda^{-1}) (c_-^{-1/\lambda}+c_+^{-1/\lambda})\right)
\ee
as $\ve\to 0+.$

Hence $-2A_+L_\ve\sim\ln\ve.$ So
$ L_\ve =(2A_+)^{-1}\ln(\ve^{-1})(1+f(\ve)),$  where $f(\ve)\to 0, \ \ve\to 0.$ Substituting this into \eqref{eq_298}, we obtain after cancellations
$$
-f(\ve)
\ln(\ve^{-1})+\lambda^{-1}(\ln\ln(\ve^{-1}) -\ln(2A_+)) -
$$
$$
-\ln\left( \Gamma(1+\lambda^{-1}) (c_-^{-1/\lambda}+c_+^{-1/\lambda})\right)=\ln\alpha_{ {+}}+o(1).
$$
Hence
$$
f(\ve)=\frac{
\lambda^{-1}(\ln\ln(\ve^{-1}) -\ln(2A_+)) -\ln\left( \Gamma(1+\lambda^{-1}) (c_-^{-1/\lambda}+c_+^{-1/\lambda})\right)-\ln\alpha_{ {+}}+o(1)}{\ln(\ve^{-1})}
$$
and we get \eqref{eq_270}.

The proof that   \eqref{eq_265} follows from \eqref{eq_270} is similar.
\end{proof}
If $\int_{-1}^1a(y)dy=0,$ then  $2A_+=\int_{-\infty}^\infty|a(y)|dy=:\|a\|_1.$
Next result follows from   Theorem \ref{thm2.1} and Corollary \ref{cor1}.
\begin{thm}\label{thm2.2}
Let $\int_{-1}^1 a(y)dy=0$ and $2A(x)\sim c_{\pm}|x|^\lambda, x\to 0 \pm,$ where $c_\pm>0,$ $\lambda>0.$

Assume that
$$
L_\ve=(\|a\|_1)^{-1}\Big(
\ln(\ve^{-1})+\lambda^{-1}\ln\ln(\ve^{-1}) {-}
$$
$$
 {-}\left(\ln(\alpha) + \ln(\Gamma(1+\lambda^{-1}))+ \lambda^{-1}\ln(\|a\|_1)
+\ln(c_+^{-1/\lambda}+c_-^{-1/\lambda})\right)
\Big) +o(1), \ \ve\to 0+,
$$
where $\alpha>0.$

 {Let $X_\ve(0)=x\neq 0, \ve>0$.}

Then  $\{X_\ve\}$ converges in distribution to the Brownian motion with a hard membrane, where $\alpha_+=\alpha_-=\alpha.$
\end{thm}
\begin{remk}
{
The parameters $\alpha_+$ and $\alpha_-$ coincide for the limiting process despite $c_\pm$ being possibly different. The corresponding formula can also be obtained in the different way that uses only ``one-sided'' arguments.} Indeed, similarly to the proof of Theorems \ref{thm_main1}, \ref{thm2.1} it can be shown that if
\be\label{eq_beta}
L_\ve=(2A_+)^{-1}\Big(
\ln(\ve^{-1})+\lambda^{-1}\ln\ln(\ve^{-1}) {-}
\ee
$$
 {-}\left(\ln \beta  + \ln(\Gamma(1+\lambda^{-1}))+  \lambda^{-1}\ln(2A_+)\right) \Big) +o(1), \ \ve\to 0+,
$$
then
\be\label{eq_306}
\lim_{\ve\to 0+}\ve^{-1}P\left(\zeta^{(\ve)}_0<\zeta^{(\ve)}_{2\ve} \ | \ X_\ve(0)=\ve\right)= c_+^{1/\lambda}\beta,
\ee
where $\zeta_r^{(\ve)}=\inf\{ t\geq 0:\ \ X_{\ve}(t)=r\}.$ Moreover
\be\label{eq_312}
X_\ve(\cdot\wedge \zeta^{(\ve)}_0) \Rightarrow
 w^{refl, +}(\cdot\wedge \zeta( {c_+^{1/\lambda}}\beta)),
\ee
where $\zeta( {c_+^{1/\lambda}}\beta)= \inf\{t\geq 0: \ l^+(t)\geq Exp( {c_+^{1/\lambda}}\beta)\}.$

It can be proved (again similarly to the proof of Theorem \ref{thm2.1}) that
\be\label{eq_317}
\lim_{\ve\to 0+}  P\left(\zeta^{(\ve)}_{-2\ve}<\zeta^{(\ve)}_{ 2\ve}  \ | \ X_\ve(0)=0\right) =
\frac{c_+^{-1/\lambda}}{c_+^{-1/\lambda}+c_-^{-1/\lambda}}.
\ee
Hence, if  {
$c_+^{1/\lambda}\frac{c_+^{-1/\lambda}}{c_+^{-1/\lambda}+c_-^{-1/\lambda}}\beta=\alpha,$ i.e., $\beta=(c_+^{-1/\lambda}+c_-^{-1/\lambda})\alpha,$
then
  \eqref{eq_306}, \eqref{eq_312}, \eqref{eq_317},  and memorylessness of exponential distribution  imply  Theorem \ref{thm2.2}.
  }
\end{remk}
\begin{remk} %Let us consider possible limits of   $\{X_\ve\}$
Suppose that some conditions of Theorem \ref{thm2.2} are not satisfied. Let $ X_\ve(0)=x>0$. It can be seen that in each of the following cases $X_\ve\Rightarrow w^{refl, +} $ as $  \ve\to 0+$

\begin{itemize}

\item $L_\ve$ has a form \eqref{eq_beta} and either $A_+>A_-$ or   $  2A(x)\sim c_{\pm}|x|^{\lambda_\pm}, x\to 0 \pm$, where $c_\pm >0$ and $0<\lambda_+<\lambda_-.$

\item   $2A(x)\sim c_{+} x^{\lambda_+}, x\to 0 + $   and for some $\delta>0$
$$
L_\ve\geq (2A_+)^{-1} \left(
\ln(\ve^{-1})+(1+\delta)\lambda^{-1}\ln\ln(\ve^{-1})\right) +o(1).
$$

\end{itemize}

\end{remk}

 {
{\bf Acknowledgement}. The authors thank  anonymous referees for
careful reading and pointing out  authors' oversights.}

\end{document}